\documentclass[reqno,b5paper]{amsart}

\usepackage{amssymb}
\usepackage{dsfont}
\usepackage{color}

\newtheorem{theorem}{Theorem}[section]

\newtheorem{proposition}[theorem]{Proposition}
\newtheorem{corollary}[theorem]{Corollary}
\newtheorem{lemma}[theorem]{Lemma}
\newtheorem{example}[theorem]{Example}
\newtheorem{remark}[theorem]{Remark}
\newtheorem{definition}[theorem]{Definition}
\newtheorem{fact}[theorem]{Fact}

\newcommand{\G}{\mathbb{G}}
\newcommand{\T}{\mathbb{T}}
\newcommand{\Z}{\mathbb{Z}}
\newcommand{\iZ}{\mathcal{Z}}
\newcommand{\N}{\mathbb{N}}

\newcommand{\supp}{\mathrm{supp}}

\catcode`\@=12
\def\T{{\mathbb T}}
\def\eps{{\varepsilon}}

\def\Z{{\mathbb Z}}
\def\N{{\mathbb N}}
\def\R{{\mathbb R}}
\def\Q{{\mathbb Q}}

\makeindex

\begin{document}

\title[Generating subgroups of the circle]
{Generating subgroups of the circle using a generalized class of density functions}

\author{Pratulananda Das}

\address{Department of Mathematics, Jadavpur University, Kolkata-700032, India}
\email {pratulananda@yahoo.co.in}

\author{Ayan Ghosh}

\address{Department of Mathematics, Jadavpur University, Kolkata-700032, India}
\email {ayanghosh.jumath@gmail.com}

\begin{abstract}

In this article, we  consider the generalized version $d^f_g$ of the natural density function introduced in \cite{BDK} where $g : \N \rightarrow [0,\infty)$
satisfies $g(n) \rightarrow \infty$ and $\frac{n}{g(n)} \nrightarrow 0$ whereas $f$ is an unbounded modulus function and generate versions of characterized subgroups of the circle group $\T$ using these density functions. We show that these subgroups have the same feature as the $s$-characterized subgroups \cite{DDB} or $\alpha$-characterized subgroups \cite{BDH} and our results provide more general versions of the main results of both the articles. But at the same time the utility of this more general approach is justified by constructing new and nontrivial subgroups for suitable choice of $f$ and $g$. In several of our results we use properties of the ideal $\iZ_g(f)$ which are first presented along with certain new observations about these ideals which were not there in \cite{BDK}.

\end{abstract}
\maketitle
%
%

\section{Introduction and background}

Throughout this paper $\R$, $\Q$, $\Z$ and $\N$ will stand for the set of all real numbers, the
set of all rational numbers, the set of all integers and the set of all natural
numbers respectively. The first three are equipped with their usual abelian
group structure, while we denote by $\T$ the circle group $\R/\Z$ in additive notation. For a positive natural number $m$ we denote by $\Z(m)$ the set of solutions of the equation $mx=0$ in $\T$. This is a cyclic group of order $m$. For $x\in\R$ we denote by $\lfloor x\rfloor$ the greatest integer less than $x$, and by $\{x\}$ we denote its fractional part.

Let us start with the important observation as, how historically, nice subgroups of the circle group have been generated via the notion of usual convergence using sequences of integers which later came to be known as characterized subgroups.
  \begin{definition}\cite{BDS}
Let $(a_n)$ be a sequence of integers, the subgroup
$$
t_{(a_n)}(\T) := \{x\in \T: \{a_nx\} \to 0\mbox{ in } \T\}.
$$
of $\T$ is called a characterized $($by $(a_n))$ subgroup of $\T$.
\end{definition}
  Characterized  subgroups of $\T$ have been studied widely by many authors, and a major part of it has been devoted knowing these subgroups for arithmetic sequences.   Eggleston \cite{E} observed  (see also \cite{BDMW}) that the asymptotic behavior of the
sequence $q_n := {a_n\over a_{n-1}}$
of ratios has a strong impact on the size of $t_{(a_n)}(\T)$:

(E1) $t_{(a_n)}(\T)$ is countable if $(q_n)$ is bounded;

(E2) $|t_{(a_n)}(\T)| = c$ if $q_n \to \infty$.

B\'{i}r\'{o}, Deshouillers and S\'{o}s \cite{BDS} established the important fact that every countable subgroup of $\T$ is characterized. The whole history concerning these investigations along with relevant references can be found in the surveys [\cite{D1}, \cite{DDG}] as also the recent article \cite{DDB}.

 In many cases  the subgroup  $t_{(a_n)}(\T)$
 is rather small, even if the sequence $(a_n)$ is not too dense.
This suggests that asking $a_nx \to 0 $ maybe somewhat too restrictive (as has been pointed out in more details in \cite{DDB}). A very natural instinct should be to consider modes of convergence which are more general than the notion of usual convergence and here the idea of natural density came into picture, as motivated by the above mentioned observation, Dikranjan, Das and Bose \cite{DDB} introduced the notion of  statistically characterized  subgroups of $\T$ by
relaxing the condition $a_nx \to 0 $ with the condition $a_nx \to 0 $ statistically.

For $m,n\in\mathbb{N}$ and $m\leq n$, let $[m, n]$ denotes the set $\{m, m+1, m+2,...,n\}$. By
 $|A|$ we denote the cardinality of a set $A$.
\begin{definition}(see \cite{Bu1},\cite{Bu2})
The lower and the upper natural densities of $A \subset \mathbb{N}$ are defined by
\begin{equation*}
\underline{d}(A) = \displaystyle{\liminf_{n\to\infty}}\frac{|A\cap [1,n]|}{n}
\end{equation*}
 and
\begin{equation*}
\hspace{.2cm} \overline{d}(A) = \displaystyle{\limsup_{n\to\infty}}\frac{|A\cap [1,n]|}{n}.
\end{equation*}
If $\underline{d}(A)=\overline{d}(A)$, we say that the natural
density of $A$ exists and it is denoted by $d(A)$.
\end{definition}
\begin{definition}\cite{St,F,S,Fr}
A sequence of real numbers $(x_n)$ is said to converge to a real number $x_0$ statistically if for any $\eps > 0$, $d(\{n \in \mathbb{N}: |x_n - x_0| \geq \eps\}) = 0$.
\end{definition}

 \begin{definition}\cite{DDB}
For a sequence of integers $(a_n)$ the subgroup
$$
t^s_{(a_n)}(\T) := \{x\in \T: \{a_nx\} \to 0\  \mbox{ statistically in }\  \T\}
$$
of $\T$ is called a statistically characterized (shortly, an s-characterized) $($by $(a_n))$ subgroup of $\T$.
\end{definition}

This was followed by another attempt to generate nice subgroups of $\T$ using certain kind of density function when in \cite{BDH} the notion of natural density of order $\alpha$, $0 < \alpha < 1$, (introduced in 2012 in \cite{BDP}) was used to generate corresponding characterized subgroups $t^\alpha_{(a_n)}(\T)$ and it was seen that they indeed generate, for a given arithmetic sequence $(a_n)$, new nontrivial subgroups different from $t_{(a_n)}(\T)$ as well as $t^s_{(a_n)}(\T)$.

Lately there have been certain other as also more general versions of density functions which we now recall. The modulus functions are defined as functions $f:[0,\infty)\to [0,\infty)$ which satisfy the following properties.
\begin{itemize}
\item[(i) ] $f(x)=0 \Leftrightarrow x=0$
\item[(ii) ] $f(x+y)\leq f(x)+f(y)$ for all $x,y\in (0,\infty)$         [Triangle inequality]
\item[(iii) ] $f$ is non-decreasing
\item[(iv)] $f$ is right continuous at $0$.
\end{itemize}
Note that the conditions (i)-(iv) imply the
continuity of the modulus functions which will be useful in certain  proofs. Some examples of such modulus functions \cite{AIZ} are given by
\begin{itemize}
\item[1. ] $f(x)=x, \  x\in [0,\infty)$.
\item[2. ] $f(x)=\frac{x}{1+x}, \  x\in [0,\infty)$.
\item[3. ] For any $\alpha\in (0,1)$, $f(x)=x^{\alpha}$ for $x\in [0,\infty)$.
\item[4. ] $f(x)=\log(1+x), \ x\in [0,\infty)$.
\end{itemize}
 In 2014, a notion of density function was introduced using modulus functions.  Precisely, the upper $f$ density function \cite{AIZ} was defined in the following way:
$$\overline{d}^f(A)=\limsup_{n\to\infty} \frac{f(|A\cap
[0,n-1]|)}{f(n)}.$$
Similarly the lower $f$ density function $\underline{d}^f$ is defined. As a natural consequence, the family
$$\iZ(f)=\{A\subset \N: \overline{d}^f(A)=0\}$$
forms an ideal on $\N$. Recall that a family $\mathcal{I}\subseteq \mathcal{P}(\N)$ is called an ideal on $\N$ whenever
\begin{itemize}
\item[$\bullet$] $\N\not\in \mathcal{I}$,
\item[$\bullet$] if $A,B\in\mathcal{I}$ then $A\cup B\in \mathcal{I}$,
\item[$\bullet$] if $A\subseteq B$ and $B\in\mathcal{I}$ then $A\in\mathcal{I}$.
\end{itemize}
An ideal $\mathcal{I}$ on $\N$ is tall or dense if each infinite set $A\subseteq\N$ has an infinite subset $B\in\mathcal{I}$.

On the other hand, in 2015 in \cite{BDFS} the authors defined a new class of densities using weight functions. Let $ g: \N \to \left[
{0,\infty } \right)$ be a function with $ \mathop {\lim
}\limits_{n \to \infty } g\left( n \right) = \infty .$ The upper
density of weight $g$ was defined in \cite{BDFS} by the formula
$$\overline{d}_{g}(A)= \mathop {\limsup }\limits_{n \to \infty } \frac{|A
\cap [1, n]|}{{g\left( n \right)}}$$ for $A \subset
\N$. The lower density of weight $g$, $\underline{d}_g(A)$ is defined in a similar way. Then the family
$$\iZ_{g}=\{A\subset \N: \overline{d}_{g}(A)=0\}$$
also forms an ideal. It has been observed in \cite{BDFS} that $\N
\in \iZ_{g}$ iff $\frac{n}{{g\left( n \right)}} \to 0.$
So we additionally assume that $ {n \mathord{\left/
 {\vphantom {n {g\left( n \right)}}} \right.
 \kern-\nulldelimiterspace} {g\left( n \right)}}\nrightarrow 0$ and define the set of all such weight functions $g$ by $\G$.

The approaches of \cite{BDP}, \cite{AIZ} and \cite{BDFS} were unified in \cite{BDK}. Let us now recall the definition of the density function $d^f_g$, henceforth called "muduler simple density function" (where $f$ is an unbounded modulus function) introduced in \cite{BDK} which are our prime objects in this article:
$$
\underline{d}^{f}_{g}(A)=\displaystyle{\liminf_{n\to\infty}}\frac{f(|A\cap [1,n]|)}{f(g(n))} ~~\mbox{and}~~
\overline{d}^{f}_{g}(A)=\displaystyle{\limsup_{n\to\infty}}\frac{f(|A\cap [1,n]|)}{f(g(n))}.
$$
If $\underline{d}^{f}_{g}(A)=\overline{d}^{f}_{g}(A)$, we say that $d^{f}_{g}(A)$ exists.

Following the nomenclature of \cite{BDK},
$
\iZ_{g}(f) = \{A \subset \mathbb{N}: d^{f}_{g}(A) = 0\}
$
 denotes the corresponding ideal (henceforth called moduler simple density ideal) and $\iZ_g^*(f)$ is the dual filter i.e. $\iZ_g^*(f) = \{A \subset \mathbb{N}: d^{f}_g(\N \setminus A) =0\}$. Further in view of our next result we can assume  $\G$ to consist of only non-decreasing weight functions $ g: \N \to [0,\infty)$ such that $\lim\limits_{n \to \infty } g(n) = \infty $ and $\frac{n}{g(n)}\nrightarrow 0$ without any loss of generality.
\begin{lemma} \label{gincreasing}\cite{BDK}
Let $f$ be an unbounded modulus function as specified. For each function $g\in \G$ there exists a non-decreasing function
$g'\in \G$ such that $\iZ_{g'}(f) = \iZ_{g}(f)$. Moreover, $f(g'(n))\leq f(g(n))$
for all $n\in\N$.
\end{lemma}

This article is broadly divided into two parts. In the first part, namely Section 2, we primarily continue the investigation of the ideal $\iZ_{g}(f)$ from \cite{BDK} and present several new observations all of which have interesting applications in the next two sections. In Section 3 we take a more unified approach to the recent study of characterized groups by  considering the notion of  density function  $d^f_g$. As mentioned already, all the notions of density functions, precisely, natural density $d$ \cite{Bu1}, natural density $d_\alpha$ of order $\alpha$ \cite{BDP}, their generalization with respect to weight function $g$, $d_g$ \cite{BDFS} and natural density with respect to an unbounded modulus function $f$, $f$-density $d^f$ \cite{AIZ} are special cases of ``$f$ density of weight $g$" i.e. $d^f_g$ \cite{BDK} and the consequence is that, not only the main results of \cite{DDB} and \cite{BDH} follow as special cases of our results, namely, Theorem 3.4 (extends Theorem A \cite{DDB}), Theorem 3.12 (extending Theorem B \cite{DDB}) and Theorem 3.13 (extending Theorem C \cite{DDB}), at the same time, the questions about $f$-density are resolved. Finally the justification for the investigation is assured by  Theorem 4.1 and Theorem 4.2 (proved in Section 4) which shows that for a given arithmetic sequence, we can indeed construct non-trivial Borel subgroups of $\T$ different from $t^s_{(a_n)}(\T)$ or $t^\alpha_{(a_n)}(\T)$ for suitable choice of modulus function $f$ or the weight function $g$. The article ends with some interesting comparative results about the generated subgroups which are also presented in this section.

\section{Certain further observations concerning the ideal $\iZ_{g}(f)$}

We start with the observation that for any unbounded modulus function $f$ and $g\in\G$, the generated ideal $\iZ_{g}(f)$ contains at least one infinite subset of $\N$ as from the result mentioned below one can easily verify that $\iZ_g(f) \supsetneq Fin$ where as usual $Fin$ stands for the ideal of all finite sets.

\begin{lemma}\label{fgnk}\cite[Lemma 3.6]{BDK}
Let $f$ be an unbounded modulus function and let $g\in \G$. Further let us define a sequence $(n_k)$ inductively such that $n_0:=1$ and $n_{k+1}:=\min\{n\in\N : f(g(n))\geq 2f(g(n_k))\}$. Then
$$
\iZ_{g}(f) = \{A\subset\N : \lim\limits_{k\rightarrow\infty}\frac{f(|A\cap[n_k,n_{k+1})|)}{f(g(n_k))}=0\}.
$$
\end{lemma}
For any unbounded modulus function $f$ and $g\in\G$ it is obvious that $\iZ_g(f)\supseteq Fin$. Set $A=\{m_k\in\N : m_k\in [n_k,n_{k+1})\}$ and note that
$$
d^f_g(A) = \lim\limits_{k\rightarrow\infty}\frac{f(|A\cap[n_k,n_{k+1})|)}{f(g(n_k))} = \lim\limits_{k\rightarrow\infty}\frac{f(1)}{f(g(n_k))} = 0.
$$
Therefore there always exists an infinite $A\subset\N$ such that $A\in\iZ_g(f)\setminus Fin$.

In general ideals generated by some kind of density function are called density ideals. However here we will use the following definition in the sense of Farah (\cite{Fa}, Definition 1.13.1, p 42). For a positive measure $\mu$ defined on subsets of $\N$, the support of $\mu$ is the set
$\{n \in \N : \mu(\{n\})> 0\}$. We say that an ideal $\mathcal{I}$ on $\N$ is a density ideal if $\mathcal{I} = Exh(\varphi) = \{ A \subset \N\colon \lim_{n\to\infty}\varphi(A
\setminus [1,n])=0\}$ where $\varphi := \sup_{i \in \N} \mu_i$ and $\mu_i$ are positive measures with pairwise disjoint supports being finite subsets of $\N$. We then say that $\mathcal{I}$ is a density ideal generated by the sequence $(\mu_i)_{i \in \N}$ (note that the set $Exh(\varphi)$ may not be an ideal in general). Further recall that an ideal $\mathcal{I}$ is called dense if for every infinite $A \subset \N$ there is an infinite $B \in \mathcal{I}$ such that $B \subset A$.
\begin{proposition}\label{tall}
For any unbounded modulus function $f$ and for any $g\in\G$, the ideal $\iZ_g(f)$ is tall or dense.
\end{proposition}
\begin{proof}
Let $f$ be an unbounded modulus function and $g\in\G$. Then, from \cite[Theorem 3.7]{BDK} the density ideal $\iZ_{g}(f)$ is generated by the sequence of measures $(\mu_{k})$ given by
$$
\mu_{k}(A)=\frac{f(|A\cap [n_{k},n_{k+1})|)}{f(g(n_{k}))}.
$$
Therefore
$$
\lim\limits_{k\to\infty}\sup\limits_{i\in\N} \mu_k(\{i\}) = \ \lim\limits_{k\to\infty}\sup\limits_{i\in\N}\frac{f(|\{i\}\cap [n_{k},n_{k+1})|)}{f(g(n_{k}))} = \ \lim\limits_{k\to\infty}\frac{f(1)}{f(g(n_{k}))} = 0.
$$
Thus, from \cite[Proposition 3.4]{BDFS}, it follows that the density ideal $\iZ_g(f)$ is tall.
\end{proof}

The following results will be needed for our next observations.

\begin{lemma}\label{result3}\cite[Lemma 3.3]{BDK}
Let $f$ be an unbounded modulus function and let $g\in \G$ be such that $\frac{f(n)}{f(g(n))}\to\infty$. Then there exists a set $A \subset\N$ such that the sequence $(\frac{f(|A \cap [1,n]|)}{f(g(n))})$ is bounded but not convergent to $0$.
\end{lemma}

\begin{lemma}\label{result4}\cite[Lemma 3.4]{BDK}
If $f_1$ and $f_2$ are two modulus functions and $g_1,g_2\in \G$ are such that there exist $c_1,c_2>0$, $k\in\N$ for which $\frac{f_1(x)}{f_2(x)}\geq c_1$ for all $x\neq 0$ and $\frac{f_1(g_1(n))}{f_2(g_2(n))}\leq c_2$ for all $n\geq k$, then $\iZ_{g_1}(f_1) \subseteq \iZ_{g_2}(f_2)$.
\end{lemma}

\begin{proposition} \label{result5}
Let $f$ be an unbounded modulus function. If $g_{1},g_{2}\in \G$ are such that $\frac{f(n)}{f(g_2(n))} \geq a>0$ and
$\frac{f(g_{2}(n))}{f(g_{1}(n))} \to \infty$ then $\iZ_{g_{1}}(f) \subsetneq
\iZ_{g_{2}}(f)$.
\end{proposition}
\begin{proof}
Taking $f_1=f_2=f$ in Lemma \ref{result4}, we get
$\iZ_{g_{1}}(f) \subseteq \iZ_{g_{2}}(f)$. We choose a function $g_3:\N\to [0,\infty]$ in such a way that
$f(g_3 (n)):=\sqrt{f(g_{1}(n))\cdot f(g_{2}(n))}$ holds for all $n \in \N$. The existence of such a function $g_3$ is assured as also this function  is well-defined as the function $f$ is non-decreasing. Now
\begin{equation}\label{E1}
 \lim\limits_{n\to \infty} \frac{f(g_{1}(n))}{f(g_3(n))} =\lim\limits_{n\to \infty} \sqrt{\frac{f(g_{1}(n))}{f(g_2(n))}} = \lim\limits_{n \to \infty} \frac{f(g_3(n))}{f(g_{2}(n))} = 0.
\end{equation}
Since $\frac{f(n)}{f(g_2(n))}\geq a$ for some $a\in (0,\infty)$ and $\frac{f(g_2(n))}{f(g_1(n))}\to\infty$, we have
\begin{eqnarray*}
 \frac{f(n)}{f(g_1(n))} &=& \frac{f(n)}{f(g_2(n))}\cdot\frac{f(g_2(n))}{f(g_1(n))}\to\infty
 \\
\Rightarrow\ \frac{f(n)}{f(g_3(n))} &=& \sqrt{\frac{{f(n)}^2}{f(g_1(n))f(g_2(n))}}\to\infty .
\end{eqnarray*}
Therefore, from Lemma \ref{result3}, there exists $A\subset\N$ such that $(\frac{f(|A \cap [1,n]|)}{f(g_3(n))})_{n\in\N}$ is bounded but not convergent to zero. We claim that $A \in \iZ_{g_{2}}(f)\setminus
 \iZ_{g_{1}}(f)$. Indeed from (\ref{E1}) we have
\begin{equation*}
\frac{f(|A \cap [1,n]|)}{f(g_{2}(n))}\ =\ \frac{f(|A \cap [1,n]|)}{f(g_3(n))} \cdot \frac{f(g_3(n))}{f(g_{2}(n))} \to 0
\end{equation*}
 whereas
\begin{equation*}
\hspace{.3cm} \frac{f(|A \cap [1,n]|)}{f(g_{1}(n))}\ =\ \frac{f(|A
\cap [1,n]|)}{f(g_3(n))} \cdot \frac{f(g_3(n))}{f(g_{1}(n))} \not\to 0.
\end{equation*}
This shows that $A \in \iZ_{g_{2}}(f)\setminus \iZ_{g_{1}}(f)$.
\end{proof}

In \cite{BDK} comparisons were made between the ideals $\iZ_g$ and $\iZ_g(f)$ showing that for suitable choice of the modulus function $f$, $\iZ_{g} \neq \iZ_g(f)$ (see Proposition 2.5 and Remark 2.6 \cite{BDK}). However no comparative study was carried out between the ideals $\iZ(f)$ and $\iZ_g(f)$ and of particular interest is the question whether there are instances where these two ideals would be different which is answered in the  following results.

\begin{corollary}
For any unbounded modulus function $f$ and $g\in \G$ if $f(n)/f(g(n))\to\infty$ then $\iZ_{g}(f)
\subsetneq \iZ(f)$.
\end{corollary}\label{cc1}
\begin{proof}
Taking $g_2(n)=n$ for all $n\in\N$ and $g_1=g$ in Proposition \ref{result5}, we obtain $\iZ_{g}(f)
\subsetneq \iZ(f)$.
\end{proof}

\begin{proposition}\label{pnsubset}
For any unbounded modulus function $f$, there exists $g\in\G$ such that $\iZ(f) \subsetneq \iZ_g(f)$.
\end{proposition}

\begin{proof}
Let $f$ be an unbounded modulus function. Since $f$ is non-decreasing, there exists a strictly increasing sequence of natural numbers $(a_n)$ such that $f(a_{n+1})>nf(a_n)$. Set $b_n=a_{4n-2}$ and $d_n=a_{4n}$ for all $n\in\N$. Now we define a function $g$ in the following way
$$
g(n)=
  \begin{cases}
   d_k & \text{for}\ b_k<n\leq d_k \\
   n & \text{for}\ d_k<n\leq b_{k+1}.
  \end{cases}
$$
 From the construction, it is evident that  $g$ is non-decreasing. Since $g(n)\to\infty$ and $\frac{d_k}{g(d_k)}=1$ for all $k\in\N$ we have $\frac{n}{g(n)}\nrightarrow 0$ and therefore $g\in\G$. Note that $g(n)\geq n$ for all $n\in\N$. As $f$ is non-decreasing, $f(g(n))\geq f(n)$ i.e. $\frac{f(n)}{f(g(n))}\leq 1$ for all $n\in\N$. Therefore, in view of Lemma \ref{result4}, we have $\iZ(f) \subseteq \iZ_g(f)$.

 Let us define $A=\bigcup\limits_{k=1}^{\infty}(b_k,c_k]$ where $c_k=a_{4k-1}<d_k$. Consider any $m\in\N$ and set $n_m=b_m$. Now for any $n>n_m$ we have either $n\in (b_{m_0},c_{m_0}]$ or $n\in (c_{m_0},d_{m_0}]$ or $n\in (d_{m_0},b_{m_0+1}]$ for some natural number $m_0\geq m$.\\ For $n\in (b_{m_0},c_{m_0}]$, we have
$$
\frac{f(|A\cap[1,n]|)}{f(g(n))}\leq \frac{f(|A\cap[1,c_{m_0}]|)}{f(d_{m_0})} \leq \frac{f(c_{m_0})}{f(d_{m_0})} < \frac{1}{4m_0-1}\leq  \frac{1}{m}.
$$
For $n\in (c_{m_0},d_{m_0}]$, we have
$$
\frac{f(|A\cap[1,n]|)}{f(g(n))}\leq \frac{f(|A\cap[1,d_{m_0}]|)}{f(d_{m_0})}= \frac{f(|A\cap[1,c_{m_0}]|)}{f(d_{m_0})} \leq \frac{f(c_{m_0})}{f(d_{m_0})} < \frac{1}{4m_0-1}\leq  \frac{1}{m}.
$$
And, for $n\in (d_{m_0},b_{m_0+1}]$ we also have
$$
\frac{f(|A\cap[1,n]|)}{f(g(n))}\leq \frac{f(|A\cap[1,b_{m_0+1}]|)}{f(n)}= \frac{f(|A\cap[1,c_{m_0}]|)}{f(n)} \leq \frac{f(c_{m_0})}{f(d_{m_0})} < \frac{1}{4m_0-1}\leq  \frac{1}{m}.
$$
As a result we can conclude that $\frac{f(|A\cap[1,n]|)}{f(g(n))}<\frac{1}{m}$ for all $n>n_m$. Since $m\in\N$ was chosen arbitrarily we obtain $\frac{f(|A\cap[1,n]|)}{f(g(n))}\to 0$ i.e. $A\in\iZ_g(f)$. But $A\not\in\iZ(f)$, because for all $k\in\N$, observe that
$$
\frac{f(|A\cap[1,c_k]|)}{f(c_k)}\geq \frac{f(c_k-b_k)}{f(c_k)}\geq \frac{f(c_k)-f(b_k)}{f(c_k)}\to 1 \mbox{    (Since, $ \lim\limits_{k\to\infty} \frac{f(a_{4k-2})}{f(a_{4k-1})}=0$)}.
$$
Thus we conclude that $\iZ(f) \subsetneq \iZ_g(f)$.
\end{proof}

\begin{remark}\label{remarkc}
 For each $A=\{n_1<n_2<\ldots<n_k<\ldots\}\subseteq\N$ we can define
$$
g_A(n)=
  \begin{cases}
   d_{n_k} & \text{for }\ b_{n_k}<n\leq d_{n_k} \\
   n & \text{otherwise}
  \end{cases}
.$$
It is easy to verify that $g_A\in\G$ and $\iZ(f) \subsetneq \iZ_{g_A}(f)$. Therefore, there exists $\mathfrak c$ many choice of $g\in\G$ for which $\iZ(f) \subsetneq \iZ_{g}(f)$.
\end{remark}

We will end this section with some more comparison results which are interesting in their own right and will also come to our use in the last section of this article. The following lemma will play the most vital role hereafter.

\begin{lemma}\label{pgeneral}
For any two unbounded modulus functions $f_1,f_2$, there exists a strictly increasing sequence $(a_n)$ of natural numbers such that $\frac{f_1(a_{n+1})}{f_1{(a_n)}}>n$, $\frac{f_2(a_{n+1})}{f_2{(a_n)}}>n$ and $a_{n+1}> 2 a_n$.
\end{lemma}
\begin{proof}
Set $a_1$=1. Choose $a_2\in\N$ in such a way that $a_2>2a_1$ while $f_1(a_2)>f_1(1),f_1(a_2)>f_2(1)$ and $f_2(a_2)>f_1(1),f_2(a_2)>f_2(1)$ are satisfied. Now inductively we define
$$
a_{n+1}=\{r\in\N: \ \min\{f_1(r),f_2(r)\}>n\max\{f_1(a_n),f_2(a_n)\} \ \mbox{ and }r>2a_n \} \ \mbox{  for all } n\in\N.
$$
Since $f_1(n)\to\infty$ and $f_2(n)\to\infty$, so $a_{n+1}$ is well defined for all $n\in\N$. From the construction, it is clear that $(a_n)$ is strictly increasing and $\frac{f_1(a_{n+1})}{f_1{(a_n)}}>n$, $\frac{f_2(a_{n+1})}{f_2{(a_n)}}>n$ and $a_{n+1}> 2 a_n$ for all $n\in\N$.
\end{proof}

Our next result is the most important observation regarding the ideals $\iZ_g(f)$ in this section (which was left out in \cite{BDK}) where we show that there exists an antichain of cardinality $\mathfrak c$ of such ideals in line of Theorem 2.7 \cite{BDFS}. In order to prove the result, we will need the following: Two sets $P,Q\subseteq\N$  are said to be almost disjoint if $P\cap Q$ is finite. Theorem 5.35 \cite{Bu} assures the existence of a family $\mathcal J$ of infinite pairwise almost disjoint subsets of $\N$ with $|\mathcal{J}| = \mathfrak c$ .

\begin{theorem}\label{tncomparableg}
For any two unbounded modulus functions $f_1,f_2$, there exists a family $\G_0\subseteq \G$ of cardinality $\mathfrak c$ such that $\iZ_g(f_i)$ is incomparable with $\iZ(f_j)$ for each $g\in\G_0$ and $i,j\in\{1,2\}$. Also $\iZ_{g_1}(f_i)$, $\iZ_{g_2}(f_j)$ are incomparable for $i,j\in\{1,2\}$ and any two distinct $g_1,g_2\in\G_0$.
\end{theorem}

\begin{proof}

Let $f_1,f_2$ be two unbounded modulus functions. In view of Lemma \ref{pgeneral}, we can find a strictly increasing sequence of natural numbers $(a_n)$ such that $\frac{f_i(a_{n+1})}{f_i(a_n)}>n$ for $i\in\{1,2\}$ and $a_{n+1}>2a_n$. Set $b_n=a_{4n-2}$ and $d_n=a_{4n}$ for all $n\in\N$. Now for any $P=\{p_1<p_2<\ldots<p_k<\ldots\}\in\mathcal{J}$, we define
$$
g_P(n)=
  \begin{cases}
   d_{p_k} & \text{for}\ b_{p_k}<n\leq b_{p_k+1} \\
   n & \text{otherwise}.
  \end{cases}
$$
Note that $g_P\in\G$. Consider the set
$$
A_P=\bigcup\limits_{k=1}^{\infty}(b_{p_k},c'_{p_k}] \ \mbox{   where, } c'_{p_k}=a_{4p_k-1}<d_{p_k}.
$$
We take any $m\in\N$ and set $n_m=b_{p_m}$. Now for any $n>n_m$ we have either $n\in (b_{p_{m_0}},b_{p_{m_0}+1}]$ or $n\in (b_{p_{m_0}+1},b_{p_{(m_0+1)}}]$ for some natural number $m_0\geq m$.\\
For $n\in (b_{p_{m_0}},b_{p_{m_0}+1}]$ and  $i\in\{1,2\}$, we have
$$
\frac{f_i(|A_P\cap[1,n]|)}{f_i(g_P(n))}\leq \frac{f_i(|A_P\cap[1,b_{p_{m_0}+1}]|)}{f_i(d_{p_{m_0}})}= \frac{f_i(|A_P\cap[1,c'_{p_{m_0}}]|)}{f_i(d_{p_{m_0}})} \leq \frac{f_i(c'_{p_{m_0}})}{f_i(d_{p_{m_0}})} < \frac{1}{4m_0-1}\leq  \frac{1}{m}.
$$
On the other hand, for $n\in (b_{p_{m_0}+1},b_{p_{(m_0+1)}}]$ and $i\in\{1,2\}$ we also have
$$
\frac{f_i(|A_P\cap[1,n]|)}{f_i(g_P(n))}\leq \frac{f_i(|A_P\cap[1,b_{p_{(m_0+1)}}]|)}{f_i(n)}= \frac{f_i(|A_P\cap[1,c'_{p_{m_0}}]|)}{f_i(n)} \leq \frac{f_i(c'_{p_{m_0}})}{f_i(d_{p_{m_0}})} < \frac{1}{4m_0-1}\leq  \frac{1}{m}.
$$
As a result we can conclude that $\frac{f_i(|A_P\cap[1,n]|)}{f_i(g_P(n))}<\frac{1}{m}$ for all $n>n_m$. Since $m\in\N$ was chosen arbitrarily we obtain $\frac{f_i(|A_P\cap[1,n]|)}{f_i(g_P(n))}\to 0$ i.e. $A_P\in\iZ_{g_P}(f_i)$ for $i\in\{1,2\}$. But $A_P\not\in\iZ(f_j)$ for $j\in\{1,2\}$, because for all $k\in\N$, observe that
$$
\frac{f_j(|A_P\cap[1,c'_{p_k}]|)}{f_j(c'_{p_k})}\geq \frac{f_j(c'_{p_k}-b_{p_k})}{f_j(c'_{p_k})}\geq \frac{f_j(c'_{p_k})-f_j(b_{p_k})}{f_j(c'_{p_k})}\to 1 \mbox{    (Since, $ \lim\limits_{k\to\infty} \frac{f_j(a_{4p_k-2})}{f_j(a_{4p_k-1})}=0$)}.
$$
Thus we conclude that $A_P\in \iZ_{g_P}(f_i)\setminus\iZ(f_j)$ where $i,j\in\{1,2\}$.

Now consider the set
$$
B_P=\bigcup\limits_{k=1}^{\infty}(c_{p_k+1},b_{p_k+1}] \ \mbox{  where, } c_{p_k+1}=b_{p_k+1}-d_{p_k}.
$$
Observe that $c_{p_k+1}=b_{p_k+1}-d_{p_k}=a_{4p_k+2}-a_{4p_k}> 3a_{4p_k}> d_{p_k}$. We take any $m\in\N$ and set $n_m=b_{p_m}$. Now for any $n>n_m$ we have either $n\in (b_{p_{m_0}},c_{p_{m_0}+1}]$ or $n\in (c_{p_{m_0}+1},b_{p_{(m_0+1)}}]$ for some natural number $m_0\geq m$.\\
For $n\in (c_{p_{m_0}+1},b_{p_{(m_0+1)}}]$ and  $j\in\{1,2\}$, we have
\begin{eqnarray*}
\frac{f_j(|B_P\cap[1,n]|)}{f_j(n)} &\leq& \frac{f_j(|B_P\cap[1,b_{p_{(m_0+1)}}]|)}{f_j(c_{p_{m_0}+1})} \\  &=&  \frac{f_j(|B_P\cap[1,b_{p_{m_0}+1}]|)}{f_j(c_{p_{m_0}+1})} \\ &\leq& \frac{f_j(p_{m_0}d_{p_{m_0}})}{f_j(c_{p_{m_0}+1})} \leq \frac{p_{m_0}f_j(d_{p_{m_0}})}{f_j(a_{4p_{m_0}+2}- a_{4p_{m_0}})} \\ &<& \frac{f_j(a_{4p_{m_0}+1})}{f_j(a_{4p_{m_0}+2})-f_j(a_{4p_{m_0}})} \leq\frac{1}{4p_{m_0}}\leq \frac{1}{4m_0}\leq\frac{1}{m}
\end{eqnarray*}
On the other hand, for $n\in (b_{p_{m_0}},c_{p_{m_0}+1}]$ and $j\in\{1,2\}$ we also have
\begin{eqnarray*}
\frac{f_j(|B_P\cap[1,n]|)}{f_j(n)} &\leq& \frac{f_j(|B_P\cap[1,c_{p_{m_0}+1}]|)}{f_j(b_{p_{m_0}})} \\  &=&  \frac{f_j(|B_P\cap[1,b_{p_{(m_0-1)}+1}]|)}{f_j(b_{p_{m_0}})} \\ &\leq& \frac{f_j(p_{(m_0-1)}d_{p_{(m_0-1)}})}{f_j(b_{p_{m_0}})} \leq \frac{p_{(m_0-1)}f_j(d_{p_{(m_0-1)}})}{f_j(a_{4p_{m_0}-2})} \\ &<& \frac{f_j(a_{4p_{(m_0-1)}+1})}{f_j(a_{4p_{m_0}-2})} \leq \frac{f_j(a_{4p_{m_0}-3})}{f_j(a_{4p_{m_0}-2})}\leq \frac{1}{4m_0-3}\leq\frac{1}{m}
\end{eqnarray*}
As a result we can conclude that $\frac{f_j(|B_P\cap[1,n]|)}{f_j(g_P(n))}<\frac{1}{m}$ for all $n>n_m$. Since $m\in\N$ was chosen arbitrarily we obtain $\frac{f_j(|B_P\cap[1,n]|)}{f_j(g_P(n))}\to 0$ i.e. $B_P\in\iZ(f_j)$ where $j\in\{1,2\}$. But $B_P\not\in \iZ_{g_P}(f_i)$ for $i\in\{1,2\}$, since for all $k\in\N$ we must have
$$
\frac{f_i(|B_P\cap[1,b_{p_k+1}]|)}{f_i(g_P(b_{p_k+1}))}\geq \frac{f_i(b_{p_k+1}-c_{p_k+1})}{f_i(d_{p_k})}=\frac{f_i(d_{p_k})}{f_i(d_{p_k})}=1.
$$
Thus $B_P\in\iZ(f_j)\setminus\iZ_{g_P}(f_i)$ where $i,j\in\{1,2\}$.

Let us next define $\G_0=\{g_P\in\G: \ P\in\mathcal{J}\}$. It has already been observed that $\iZ_g(f_i)$ is incomparable with $\iZ(f_j)$ for any $g\in\G_0$ and $i,j\in\{1,2\}$.
Now, consider any two distinct sets $P=\{p_1<p_2<\ldots<p_k<\ldots\}, \ Q=\{q_1<q_2<\ldots<q_k<\ldots\} \in\mathcal{J}$. We intend to show that $\iZ_{g_P}(f_i)$, $\iZ_{g_Q}(f_j)$ are incomparable where $i,j\in\{1,2\}$.
 We already have $A_P\in\iZ_{g_P}(f_i)$ and $A_Q\in\iZ_{g_Q}(f_j)$ where $i,j\in\{1,2\}$ and $A_Q,g_Q$ are similarly constructed as $A_P,g_P$ respectively.

In view of the fact that $P,Q$ are infinite and almost disjoint there exists $k_0\in\N$ such that for all $k>k_0$ we have
\begin{eqnarray*}
q_{j(k)} &<& p_k<q_{j(k)+1}
 \\
  \Rightarrow ~~~ q_{j(k)}+1 &\leq& p_k<q_{j(k)+1}
 \\
\Rightarrow \hspace{.3cm} b_{q_{j(k)}+1} &\leq& b_{p_k}<b_{q_{j(k)+1}}
 \\
\Rightarrow \hspace{.3cm} b_{q_{j(k)}+1} &\leq& b_{p_k}<c_{p_k}<b_{q_{j(k)+1}}.
\end{eqnarray*}
From the construction of $g_Q$, it follows that $g_Q(c_{p_k})=c_{p_k}$ and consequently for all $k\in\N$ and $j\in\{1,2\}$ we have,
$$
\frac{f_j(|A_P\cap[1,c_{p_k}]|)}{f_j(g_Q(c_{p_k}))}\geq \ \frac{f_j(c_{p_k}-b_{p_k})}{f_j(c_{p_k})}\geq \ \frac{f_j(c_{p_k})-f_j(b_{p_k})}{f_j(c_{p_k})}\to 1.
$$
This shows that $A_P\not\in\iZ_{g_Q}(f_j)$ for $j\in\{1,2\}$. Similarly we can also show that $A_Q\not\in\iZ_{g_P}(f_i)$ for $i\in\{1,2\}$. Thus, we can conclude that $\iZ_{g_1}(f_i)$ and $\iZ_{g_2}(f_j)$ are incomparable for $i,j\in\{1,2\}$ and any two distinct $g_1,g_2\in\G_0$.
\end{proof}

\begin{corollary}\label{c1ncomparable}
For any unbounded modulus function $f$, there exists $g\in\G$ such that $\iZ_g(f)$ is not comparable with $\iZ(f)$.
\end{corollary}

\begin{proposition}\label{ppnsubset}
For any unbounded modulus function $f$, there exist $g_1,g_2\in\G$ such that $\iZ_{g_1}(f)\subsetneq  \iZ(f) \subsetneq \iZ_{g_2}(f)$.
\end{proposition}

\begin{proof}
Let $f$ be an unbounded modulus function. Therefore, from Proposition \ref{pnsubset}, there exists a $g_2\in\G$ such that $\iZ(f) \subsetneq \iZ_{g_2}(f)$.

 As $\iZ(f) \neq Fin$, choose a set $A=\{n_1<n_2<\ldots<n_k<\ldots\}\subseteq\N$ such that $A\in\iZ(f)$ (i.e. $d^f(A)=0$). Consequently we have $\lim\limits_{k\to\infty}\frac{f(k)}{f(n_k)}= 0$. Next define,
$$
 g_1(n)=
 \begin{cases}
 1 & \text{ for } 1\leq n< n_1 \\
 k & \text{ for } \ n_k\leq n <n_{k+1}.
 \end{cases}
$$
It is easy to observe that  $g_1$ is non-decreasing while $\lim\limits_{n\to\infty}\frac{n}{g_1(n)}\nrightarrow 0$ follows from the fact that $\frac{n_k}{g_1(n_k)}=\frac{n_k}{k}>1$ for all $k\in\N$ i.e. $g_1\in\G$. Note that
$$
\lim\limits_{n\to\infty}\frac{f(n)}{f(g(n))}\geq \lim\limits_{k\to\infty}\frac{f(n_k)}{f(k)}=\infty.
$$
Hence from Corollary 2.6, we have $\iZ_{g_1}(f)\subsetneq  \iZ(f)$ and we are done.
\end{proof}

\section{The characterized subgroups for the density function $d^f_g$}

One can naturally think of the following general notion of convergence corresponding to the density function $d^f_g$.

\begin{definition}\label{Def1}
A sequence of real numbers $(x_n)$ is said to converge to a real number $x_0$ $f^g$-statistically if for any $\eps > 0$, $d^{f}_{g}(\{n \in \mathbb{N}: |x_n - x_0| \geq \eps\}) = 0$.
\end{definition}

As a natural consequence we can introduce our main definition of this section.

\begin{definition}
For a sequence of integers $(a_n)$ the subgroup
\begin{equation}\label{def:stat:conv}
t^{f,g}_{(a_n)}(\T) := \{x\in \T: a_nx \to 0\  f^g\mbox{-statistically in }\  \T\}
\end{equation}
of $\T$ is called {\em an $f^g$-statistically characterized} (shortly, {\em an $f^g$-characterized}) $($by $(a_n))$ {\em subgroup} of $\T$.
\end{definition}

Before we proceed to prove our results, it is necessary to recall the following basic facts and definitions in line of \cite{DDB}.

A sequence of positive integers $(a_n)$ is an arithmetic sequence if
$$1 = a_0 < a_1 < a_2 < \dots < a_n < \dots ~~\mbox{and}~a_n|a_{n+1}~ \mbox{for every}~n \in \mathbb{N}.$$
Let $(a_n)$ be an arithmetic sequence.  
 In this case the ratio, defined by $q_n=\frac{a_n}{a_{n-1}}$ for $n>0$ (so that $q_1:=a_1$), is a positive integer. Further we assume without loss of generality that the elements of the sequence are distinct, i.e. $q_n \geq 2$ for all $n\geq 1$. For arithmetic sequences, the fact that any $x\in\T$ can be represented canonically has been used in this sequel time and again. So here we recapitulate the result once.
 \begin{fact}\label{factsupp}
  For any arithmetic sequence $(a_n)$ and $x\in\T$, we can find a unique sequence of integers $\{c_n\}$ where $c_n\in [0,q_n-1]$ such that
 \begin{equation}\label{canonical:repr}
   x=\sum\limits_{n=1}^{\infty}\frac{c_n}{a_n},
 \end{equation}
    and $c_n<q_n-1$ for infinitely many $n$.
    \end{fact}
 We define the support of $x$ by
$$
supp_{(a_n)}(x) = \{n\in \mathbb{N}: c_n \neq 0\}.
$$
When no confusion is possible, we simply write $supp(x)$.

\begin{theorem}
 For any sequence of integers $(a_n)$, $t^{f,g}_{(a_n)}(\T)$ is an $F_{\sigma\delta}$ (hence, Borel) subgroup of $\T$ containing $t_{(a_n)}(\T)$.
 \end{theorem}
\begin{proof}
As the proof follows the same line of arguments as Theorem A \cite{DDB}, we only provide a brief sketch. It is easy to check that $t^{f,g}_{(a_n)}(\T)$ is a subgroup of the circle group $\T$. Let us set $U_{n,k}:= \left\{x\in \T:  \|a_nx\| > \frac{1}{k}\right\}$ for $n,k \in \N$. From Definition 3.2, one can write
\begin{eqnarray*}
t^{f,g}_{(a_n)}(\T) &=&\{x\in\T: (\forall k\in\N) \ d^{f}_{g}(\{n:  x\in U_{n,k}\})=0\} =  \bigcap\limits_{k=1}^{\infty}\left\{x\in\T: d^{f}_{g}(\{n: x\in U_{n,k}\})=0\right\}\\
&=&\bigcap\limits_{k=1}^{\infty}\left\{x\in\T: \lim\limits_{m\to\infty}\frac{f(|\{i\in\N:x\in U_{i,k}\}\cap [1,m]|)}{f(g(m))}=0\right\}\\
&=&\bigcap\limits_{k=1}^{\infty}\left\{x\in\T: (\forall j\in\N)(\exists m\in\N) ~\mbox{such that}~ \frac{f(|\{i\in\N:x\in U_{i,k}\}\cap [1,n]|)}{f(g(n))}\leq \frac{1}{j}~
~\mbox{ for all }~ n\geq m\right\}.
\end{eqnarray*}
Subsequently writing
$$
V_{k,j, n} =\left\{x\in\T: \frac{f(|\{i\in\N:x\in U_{i,k} \}\cap [1,n]|)}{f(g(n))}\leq \frac{1}{j}\right\}
$$
one can show that $V_{k,j, n}$ is closed in $\T$ for every  fixed triple $k, j$ and $n$. The assertion then follows from the equality
$$
t^{f,g}_{(a_n)}(\T) =\bigcap\limits_{k=1}^{\infty}\bigcap\limits_{j=1}^{\infty}\bigcup\limits_{m=1}^{\infty}\bigcap\limits_{n\geq m}
V_{k,j, n}.
$$
\end{proof}
Clearly the non-triviality of the newly obtained subgroups $t^{f,g}_{(a_n)}(\T)$ depends on
 \begin{itemize}
\item[(i)] whether $t^{f,g}_{(a_n)}(\T)$ actually becomes the whole circle group $\T$ and
\item[(ii)]  whether as subgroups of $\T$, they are really `new' compared to the already studied characterized subgroups $t_{(a_n)}(\T)$ or their versions $t^s_{(a_n)}(\T)$ and $t^\alpha_{(a_n)}(\T)$.
\end{itemize}
 The study of the first question (i) is easy, as it is known that $t_{(a_n)}(\T)= \T$ precisely when $a_n=0$ for almost all $n$ \cite{BDMW,DPS}.
 Using this fact one can conclude that $t^{f,g}_{(a_n)}(\T)= \T$ precisely when $d^f_g(\{n:a_n\neq 0\})=0$. Since no arithmetic sequence $(a_n)$ satisfies $d^f_g(\{n:a_n\neq 0\})=0$, we deduce that $t^{f,g}_{(a_n)}(\T)\ne \T$
for such sequences.

The second question (ii) is far more complicated and seems worth studying. We thoroughly investigate this problem for general arithmetic sequences.

As the general case seem quite complicated, so as in \cite{DDB} we begin with a special case providing a basic example considering the sequence $(2^n)$ and then step-by step, generalize the idea.

\vspace{0.32cm}
\subsection{The $f^g$-characterized subgroup for the sequence $\mathbf{a_n=2^n}$.}

\vspace{0.32cm}
Note that $t_{(2^n)}(\T)$ is simply the Pr\" ufer group $\Z(2^\infty)$. So it remains only to check that $t^{f,g}_{(2^n)}(\T)$ contains an element $x$ that does not belong to $ \Z(2^\infty)$. It is known that $x\in  \Z(2^\infty)$ precisely when $supp(x)$ is finite (see \cite{DI1}). Note that, for $a_n=2^n$, $c_n$ can only be $0$ or $1$. Below we take $f(x)=\log (1+x)$, $g(n)=n^{\frac{1}{2}}$ and construct an element which belongs to $t^{f,g}_{(2^n)}(\T)\setminus t_{(2^n)}(\T)$.
\begin{example}\label{example01}

Choose $x\in \T$ with
\begin{equation}\label{def:supp2^n}
\supp_{(2^n)}(x) = \bigcup\limits_{n=1}^{\infty} [(2n-1)^{(2n-1)}, (2n)^{(2n)}].
\end{equation}
 We  will show that $x\in t^{f,g}_{(2^n)}(\T) \setminus t_{(2^n)}(\T)$. To check that $x \in  t^{f,g}_{(2^n)}(\T)$, pick an $m \in \N$ and define a subset $A $ of $\N$ as follows:

\medskip

First let  $ B_n := [(2n-1)^{(2n-1)}, (2n)^{(2n)}] $. Clearly length of $B_n$ diverges to $ \infty $. Consequently one can  choose  $n_0 \in \N$  such that  $(2n_0)^{(2n_0)} - (2n_0-1)^{(2n_0-1)} > m$. Now let
$$
A_0 := [(2n_0)^{2n_0}-m, (2n_0)^{2n_0}]\mbox{ and }A_0' := [(2n_0+1)^{(2n_0+1)}-m, (2n_0+1)^{(2n_0+1)}].
$$
Similarly, let

$$
A_k := [(2(n_0+k))^{(2(n_0+k))}-m, (2(n_0+k))^{(2(n_0+k))}]$$  and $$ A_k' := [(2(n_0+k)+1)^{(2(n_0+k)+1)}-m, (2(n_0+k)+1)^{(2(n_0+k)+1)}].
$$

Finally, put  $B = \displaystyle{\bigcup\limits_{k=0}^{\infty}} (A_k\cup A_k')$ and $A=B\cup [1,(2n_0-1)^{(2n_0-1)}]$.
Note that $|A_k| = |A_k'| = m+1 $,  and so
\begin{eqnarray*}
\overline{d}^f_g(A)&=&\limsup\limits_{n\rightarrow\infty}\frac{f(|A\cap [1,n]|)}{f(g(n))}\\
&=&\max  \left\{\lim\limits_{k\rightarrow\infty}\frac{f(2(m+1)(k+1))}{f((2(n_0+k)+1)^{\frac{1}{2}\cdot (2(n_0+k)+1)})},\lim\limits_{k\rightarrow\infty}\frac{f(2(m+1)(k+1)-(m+1))}{f((2(n_0+k))^{\frac{1}{2}\cdot (2(n_0+k))})}\right\}\\
&=& \max  \left\{\lim\limits_{k\rightarrow\infty}\frac{\log (1+2(m+1)(k+1))}{\log (1+(2(n_0+k)+1)^{\frac{1}{2}\cdot (2(n_0+k)+1)})},\lim\limits_{k\rightarrow\infty}\frac{\log (1+(2(m+1)(k+1)-m))}{\log (1+(2(n_0+k))^{\frac{1}{2}\cdot (2(n_0+k))})}\right\} \ = 0 .
\end{eqnarray*}
We claim that $\|2^nx\| < 1/2^m$ for all $n \in \N \setminus A$.
 As $n \in \N \setminus A$, by the choice of $A$ and the definition of $B_n := [(2n-1)^{(2n-1)}, (2n)^{(2n)}]$, we can deduce that

 \begin{itemize}
   \item[(a)] either $n\in [(2r)^{(2r)}+1,(2r+1)^{(2r+1)}-m-1]$ for some $r\in \N$ which automatically implies that $n+1, n+2, \ldots , n+m \in [(2r)^{(2r)}+1,(2r+1)^{(2r+1)}-1]$ i.e. $n+1, n+2, \ldots , n+m\not\in \supp_{(2^n)}(x)$, or
   \item[(b)] $n\in [(2r+1)^{(2r+1)}+1,(2(r+1))^{(2(r+1))}-m-1]$ for some $r\in \N$ i.e. $n+1, n+2, \ldots , n+m \in [(2r+1)^{(2r+1)},(2(r+1))^{(2(r+1))}] \subset \supp_{(2^n)}(x)$.
 \end{itemize}

 In both cases we have $c_{n+1} =c_{n+2} = \ldots =c_{n+m}$. In case (a) this leads to $c_{n+1} =c_{n+2} = \ldots =c_{n+m}=0$ which implies
 $$
 2^nx = \frac{c_{n+1} }{2} + \frac{c_{n+2} }{2^2} + \ldots + \frac{c_{n+m} }{2^m} + \frac{c_{n+m+ 1} }{2^{m+1}} + \ldots =
  \frac{c_{n+m+ 1} }{2^{m+1}} + \ldots.
 $$
 Therefore $\|2^nx\| < 1/2^m$.
 In case (b) this leads to $c_{n+1} =c_{n+2} = \ldots =c_{n+m}=1$,  and subsequently
 $$
 2^nx = \frac{1}{2} + \frac{1}{2^2} + \ldots + \frac{1}{2^m} + \frac{c_{n+m+ 1} }{2^{m+1}} + \ldots =
1 - \frac{1}{2^m}+ \frac{c_{n+m+ 1} }{2^{m+1}} + \ldots.
$$
  As a result we can again conclude that $\|2^nx\| < 1/2^m$. Since $m\in \N$ was chosen arbitrarily and $\N \setminus A \in \iZ_g^*(f)$, we obtain that $(2^nx)$ $f^g$-statistically converges to 0 in $\T$ i.e. $x\in t^{f,g}_{(2^n)}(\T)$.
 According to \cite{DI1}, $x\notin t_{(2^n)}(\T)$ as $supp(x)$ is infinite.
\end{example}

However, we can actually prove that the newly obtained subgroup $t^{f,g}_{(2^n)}(\T)$ contains uncountably more elements compared to $t_{(2^n)}(\T)$ as had been observed for $t^s_{(2^n)}(\T)$ (Proposition 3.5 \cite{DDB}). We prove that in Proposition \ref{cor2}.

  Now we are in a position to see that the element $x\in \T$ in Example \ref{example01} can be replaced by a more generally defined element of $\T$ without any restriction on $f$ or $g$. To explain the choice, we note that for every $x$ as in (\ref{def:supp2^n}) such that $x\not \in \Z(2^\infty)$, the support can be presented as a disjoint union of infinitely many consecutive intervals $\displaystyle{\bigcup_n} B_n$. Let us define
\begin{equation}\label{II}
\mathbb{I}^f_g=\left\{\bigcup\limits_{r=1}^{\infty}B_r:    B_r=[n_{(2r-1)},n_{(2r)}] \mbox{, for some } A=\{n_r \}_{r\in \N} \subset \N \mbox{ with } d^f_g(A) \ =0 \right\}.
\end{equation}
 In  Example \ref{example01} we used the following specific member of $\mathbb{I}^f_g$
\begin{equation}\label{EqJune26}
 B = \bigcup\limits_{r=1}^{\infty}B_r \in \mathbb{I}^f_g, \mbox{ with }B_r := [(2r-1)^{(2r-1)}, (2r)^{(2r)}].
\end{equation}

 Now we intend to show that $\mathbb{I}^f_g \not \subseteq \iZ_g(f)$. If possible let us assume that $\mathbb{I}^f_g \subseteq \iZ_g(f)$ for some unbounded  modulus function $f$ and $g\in \G$. Note that for any unbounded modulus function $f$ and for any $g\in\G$, we have $\iZ_g(f)\neq Fin$. Therefore, we can choose $A=\{n_1 < n_2 < n_3 < \dots\} \subset \N $ such that $I_A=\bigcup\limits_{r=1}^{\infty}B_r \in \mathbb{I}^f_g \subseteq \iZ_g(f)$, where $B_r=[n_{(2r-1)},n_{(2r)}]$. Then, for $A'= (n_{r+1}) \subseteq A$, we have $I_{A'}=\bigcup\limits_{r=1}^{\infty}{B'}_r \in \mathbb{I}^f_g \subseteq \iZ_g(f)$, where ${B'}_r=[n_{(2r)},n_{(2r+1)}]$. But this implies that $\N \in \iZ_g(f)$ which is a contradiction. Our next observation is a result concerning both  $\mathbb{I}^f_g$ and $\iZ_g(f)$ in line of (Lemma 3.3 \cite{DDB}) that will be frequently used in the sequel.

\begin{lemma}\label{cor1}
$|\mathbb{I}^f_g|= |\iZ_g(f)|=\mathfrak c$.
\end{lemma}
\begin{proof}

Fix a specific member $B = \bigcup\limits_{r=1}^{\infty}B_r \in \mathbb{I}^f_g$, e.g., as in (\ref{EqJune26}).  Fix a sequence $\xi = (z_i)\in \{0,1\}^\N$ and define $B^\xi = \bigcup\limits_{k=1}^{\infty}B_{2k+ z_k}$. In other words, this subset $B^\xi $ of $B$ is obtained by taking at each stage $k$ either  $B_{2k}$ of $B_{2k+1}$ depending on the choice imposed by $\xi$. As obviously $B^\xi \ne B^\eta$
for distinct $\xi, \eta \in \{0,1\}^\N$, this provides an injective map given by
$$
\{0,1\}^\N \ni \xi \to B^\xi \in \mathbb{I}^f_g,
$$
Since $|\{0,1\}^\N| = \mathfrak c$, we are done.

  A similar proof works for $\iZ_g(f)$.
\end{proof}

Let us note that the element $x\in \T$ in Example \ref{example01} has the property $supp(x)\in \mathbb{I}^f_g$. Now we see that the argument works with any element $x$ of $\T$ with $supp(x)\in \mathbb{I}^f_g$ where $f$ is an unbounded modulus function and $g \in \G$.

\begin{lemma}\label{lemma1}
Let $x\in\T$ be such that $supp_{(2^n)}(x)\in \mathbb{I}^f_g$. Then $x\in t^{f,g}_{(2^n)}(\T)\setminus t_{(2^n)}(\T)$.
\end{lemma}

 \begin{proof} The fact that $x\notin t_{(2^n)}(\T)$ follows from the fact that $supp(x)\in\mathbb{I}^f_g$ implies $supp(x)$ is infinite.

 We take $\supp_{(2^n)}(x) = \bigcup\limits_{r=1}^{\infty} [n_{(2r-1)},n_{(2r)}]$, where $A'=(n_r) \in\iZ_g(f)$.
  Let us define $B_r=[n_{(2r-1)},n_{(2r)}]$ and $G_r := [n_{(2r)}+1,n_{(2r+1)}-1]$. Consider any $m \in \N$. We choose $r_0\in \N$ such that $n_{r_0} > m$ . Now we define, $A_0 =\{ n_r :\ r \in \N$ and $ r \geq r_0 \}$ and $A_i=\{ n_r -i :\ r \in \N$ and $ r \geq r_0 \}\cap\N$.
  Consequently
  \begin{eqnarray*}
  d^f_g(A_i) &=& \lim\limits_{n\rightarrow\infty} \frac{f(|(A_0-i)\cap [1,n]|)}{f(g(n))}\\
   &\leq& \lim\limits_{n\rightarrow\infty} \frac{f(|A_0 \cap [1,n]|+i)}{f(g(n))} \leq \ d^f_g(A')+ \lim\limits_{n\rightarrow\infty} \frac{f(i)}{f(g(n))} \ = 0
  \end{eqnarray*}
 Finally put $A= \bigcup\limits_{i=0}^{m} A_i \cup [1, n_{r_0}]$. We can then show that this $A$ witnesses the needed $f^g$-statistical convergence with respect to $\eps=1/2^m$ following the line of the proof of Example \ref{example01}.
 \end{proof}

Immediately we have the following result.
\begin{proposition}\label{cor2}
$|t^{f,g}_{(2^n)}(\T)\setminus t_{(2^n)}(\T)|=\mathfrak{c}$.
\end{proposition}
\begin{proof}
 In Lemma \ref{lemma1} we have shown  that $\{x:supp_{(2^n)}(x)\in \mathbb{I}^f_g\}\subset t^{f,g}_{(2^n)}(\T)\setminus t_{(2^n)}(\T)$. Now as $|\{x:\ supp_{(2^n)} (x)\in \mathbb{I}^f_g\}|=|\mathbb{I}^f_g|$, so Lemma \ref{cor1} tells us that $|\{x:supp_{(2^n)} (x)\in \mathbb{I}^f_g\}|=|\mathbb{I}^f_g|=\mathfrak{c}.$
That is, $|t^{f,g}_{(2^n)}(\T)\setminus t_{(2^n)}(\T)|\geq \mathfrak{c}$ which gives our result.

\end{proof}
On the basis of the existing knowledge that $t_{(2^n)}(\T)$ is countably infinite in size, an obvious but important consequence coming from Proposition \ref{cor2} is the following.
\begin{corollary}\label{cor10}
$|t^{f,g}_{(2^n)}(\T)|=\mathfrak{c}$.
\end{corollary}
\begin{proof}
The observation immediately follows as
$$
 t^{f,g}_{(2^n)}(\T)\setminus t_{(2^n)}(\T)\subseteq t^{f,g}_{(2^n)}(\T)\subseteq \T.
 $$.
\end{proof}

\subsection{The general case for arithmetic sequences and some more observations}

In this section, we generalize the whole idea of the last section for arbitrary arithmetic sequences and try to generalize Example \ref{example01} and Corollary \ref{cor10} in this context.

First we prove a lemma analogous to Lemma \ref{lemma1} which gives a sufficient condition for some $x$ to be in $t^{f,g}_{(a_n)}(\T)$.

\begin{lemma}\label{theorem1}
Let  $(a_n)$ be an arithmetic sequence and let $x\in\T$ be such that $supp(x)\in\mathbb{I}^f_g$ and $c_n=q_n-1$ for all $n\in supp(x)$. Then $x\in t^{f,g}_{(a_n)}(\T)$.
\end{lemma}

\begin{proof} Let $x=\sum\limits_{n=1}^{\infty} \frac{c_n}{a_n}$ be the canonical representation of $x\in\T$ where $c_1=0$, $c_n$ is either $0$ or $(q_n-1)$ for any $n>1$ and $\{n:c_n=q_n-1\}=\bigcup\limits_{r=1}^{\infty} B_r \in \mathbb{I}^f_g$ where as in \ref{lemma1} $B_r=[n_{(2r-1)},n_{(2r)}]$ and $G_r := [n_{(2r)}+1,n_{(2r+1)}-1]$ for some infinite $B= (n_r) \subseteq \N$ (i.e. the sets $B_r$ and $G_r$, $r \in \N$ forming a partition of $\N$). To show that $x\in t^{f,g}_{(a_n)}(\T)$ we proceed exactly as in Lemma \ref{lemma1}. We take an arbitrary $m\in\N$ and get the same $r_0\in\N$ and $A\subset \N$ with $d^f_g(A)=0$ where as before $A= \bigcup\limits_{i=0}^{m} A_i \cup [1, n_{r_0}]$, $A_i=\{ n_r -i :\ r \in \N$ and $ r \geq r_0 \}\cap\N$. What is required now is to show that $\lim\limits_{\stackrel{n\rightarrow\infty}{n\in\N\setminus A}}\|a_nx\|=0$. For $n \in \N \setminus A$, by the choice of $A$ and the definition of $B_r$ and $G_r$, we deduce that
  \begin{itemize}
  \item[(a)] either $n\in B_r$ for some $r\in \N$, which actually means that $ n \in [n_{2r-1}+1, n_{2r} - m-1]$ and consequently $n+1, n+2, \ldots , n+m \in B_r$, or
  \item[(b)] $n\in G_r$ for some $r\in \N$, and by the same reasoning as above we can again conclude that $n+1, n+2, \ldots , n+m \in G_r$.
 \end{itemize}
 In case (b) this leads to $c_{n+1} =c_{n+2} = \ldots =c_{n+m}=0$, and consequently
  $$
  a_nx = \sum\limits_{k=n+1+m}^{\infty}\frac{c_k}{a_k}\cdot a_n\leq \sum\limits_{k=n+1+m}^{\infty}\frac{q_k-1}{a_k}\cdot a_n=\sum\limits_{k=n+1+m}^{\infty}\left(\frac{1}{a_{k-1}}-\frac{1}{a_{k}}\right)\cdot a_n\leq \frac{a_n}{a_{m+n}}.
 $$
 In case (a) this leads to $c_k=q_k-1$ for $k=n+1,n+2,...,n+m$ which implies that
  $$
 a_nx = \sum\limits_{k=n+1}^{n+m}\frac{q_k-1}{a_k}\cdot a_n+\sum\limits_{k=n+1+m}^{\infty}\frac{c_k}{a_k}\cdot a_n.
 $$
Now the first part is
$$
\sum\limits_{k=n+1}^{n+m}\frac{q_k-1}{a_k}\cdot a_n=\sum\limits_{k=n+1}^{n+m}\left(\frac{1}{a_{k-1}}-\frac{1}{a_{k}}\right)\cdot  a_n=1-\frac{a_n}{a_{n+m}}
$$
and the second part
$$
\sum\limits_{k=n+1+m}^{\infty}\frac{c_k}{a_k}.a_n \leq \sum\limits_{k=n+1+m}^{\infty}\frac{q_k-1}{a_k}\cdot a_n=\sum\limits_{k=n+1+m}^{\infty}\left(\frac{1}{a_{k-1}}-\frac{1}{a_{k}}\right)\cdot a_n\leq \frac{a_n}{a_{m+n}}.
$$
 Therefore, we obtain that $\|a_nx\| \leq \frac{a_n}{a_{n+m}}\leq \frac{1}{2^m}$. As $m\in\N$ was chosen arbitrarily, so we conclude that $\|a_nx\|$ converges $f^g$-statistically to $0$.
 \end{proof}

Now we are going to provide another, very natural, sufficient condition for $x\in t^{f,g}_{(a_n)}(\T)$ in line of (Theorem 4.4 \cite{DDB}).

 \begin{theorem}\label{propLaaaast}
Let $(a_n)$ be an arithmetic sequence and $x\in \T$. If $ \ d^f_g (supp(x))=0$, then $x\in t^{f,g}_{(a_n)}(\T)$.
\end{theorem}

\begin{proof}
Set $A = supp(x)$. Then $d^f_g(A) = 0$, by hypothesis. Pick a positive $k\in \N$ and note that for any $i\in\N$ ,
\begin{eqnarray*}
d^f_g(A-i)= \lim\limits_{n\rightarrow\infty} \frac{f(|(A-i)\cap [1,n]|)}{f(g(n))} &\leq& \lim\limits_{n\rightarrow\infty} \frac{f(|A \cap [1,n]|+i)}{f(g(n))} \\ &\leq& \ d^f_g(A) +\lim\limits_{n\rightarrow\infty} \frac{f(i)}{f(g(n))} \ = 0.
\end{eqnarray*}
 Therefore, $A^*=\bigcup_{i=0}^k (A-i)\cap \N\in \iZ_g(f)$. Hence, it is enough to check that $\|a_nx\|\leq \frac{1}{k}$ for all $n\in \N\setminus A^*$. Note that $n\in \N\setminus A^*$ precisely when $n+i \not \in A$ for $i=0,1,\ldots k$. This means that in the canonical representation of $x$ one has $c_n=c_{n+1}= \ldots = c_{n+k}= 0$ for all $n\in \N\setminus A^*$. Hence,
\begin{eqnarray*}
\{a_nx\} =a_n\cdot\sum\limits_{i=n+k+1}^{\infty}\frac{c_i}{a_i} &\leq& a_n\cdot\sum\limits_{i=n+k+1}^{\infty}\frac{q_i-1}{a_i} \\ &=& a_n\cdot \sum\limits_{i=n+k+1}^{\infty}\left(\frac{1}{a_{i-1}}-\frac{1}{a_i}\right)  \leq \frac{a_n}{a_{n+k}}\leq \frac{1}{2^k}<\frac{1}{k}.
\end{eqnarray*}
\end{proof}

 We shall invert this theorem in Corollary \ref{Last:Corollary}. As mentioned in the introduction following is the more general version of Theorem B \cite{DDB}.
 \begin{theorem}
  For any arithmetic sequence $(a_n)$, We have $|t^{f,g}_{(a_n)}(\T)|=\mathfrak c$.
\end{theorem}
\begin{proof}
 Clearly $t^{f,g}_{(a_n)}(\T)\subset \T$ implies $|t^{f,g}_{(a_n)}(\T)|\leq |\T|=\mathfrak c$.

To prove the inequality $|t^{f,g}_{(a_n)}(\T)|\geq |\T|=\mathfrak c$ we use two alternative arguments.

Let $B\in \mathbb{I}^f_g $. Define $x_B\in \T$ with $supp(x_B)=B$ and  $c_n=q_n-1$ for all $n \in B$.
According to Lemma \ref{theorem1}, 
$x_B\in t^{f,g}_{(a_n)}(\T)$.
Since the map $\mathbb{I}^f_g \ni B  \mapsto x_B\in t^{f,g}_{(a_n)}(\T)$ is obviously injective, $|t^{f,g}_{(a_n)}(\T)|=\mathfrak c$ by Lemma \ref{cor1}.

The second argument uses the fact that $|\iZ_g(f)|= \mathfrak c$  as has been shown in Proposition \ref{cor1}.
This provides $\mathfrak c$ many elements $\{x_i:i\in I\}$ in $\T$ with distinct supports of $f^g$-density 0 and as a result
applying Theorem \ref{propLaaaast}, we obtain that $x_i \in t^{f,g}_{(a_n)}(\T)$ for every $i \in I$.
\end{proof}

Below we have the more general version of Theorem C \cite{DDB}.

\begin{theorem}
$t^{f,g}_{(a_n)}(\T)\neq t_{(a_n)}(\T)$ for any arithmetic sequence $(a_n)$.
\end{theorem}
\begin{proof}
If $\left(q_n\right)$ is bounded then $t^{f,g}_{(a_n)}(\T)\neq t_{(a_n)}(\T)$, as $t_{(a_n)}(\T)$ is countable.

Therefore, we consider $\left(q_n\right)$ is not bounded. Then there exists $B\subset \N$ such that $(q_n)_{n\in B}$ diverges to $\infty$. Now, in view of Proposition \ref{tall} there exists a $B'\subseteq B$ such that $d^f_g(B')=0$. So, in addition we can assume that $d^f_g(B) = 0$. Take
$$
 x=\sum\limits_{n=1}^{\infty}\frac{c_n}{a_n}\in\T \mbox{ with }supp(x)=B\mbox{ and }c_n=\left\lfloor \frac{ q_n}{2}\right\rfloor
 \mbox{ for all } n\in B.
$$
Then $x\in t^{f,g}_{(a_n)}(\T)$ by Theorem \ref{propLaaaast}, while $x\not\in t_{(a_n)}(\T)$ (by \cite[Theorem 2.3]{DI1}).
This proves $t^{f,g}_{(a_n)}(\T)\neq t_{(a_n)}(\T)$.

\end{proof}

 It has already been showed that for any arithmetic sequence $(a_n)$, the condition in Theorem \ref{propLaaaast} is not necessary for some $x\in\T$ to be in $t^{f,g}_{(a_n)}(\T)$ (see Example 4.5 \cite{DDB}). More precisely we have the following.

\begin{example}\em{
We have already shown that $\mathbb{I}^f_g \not \subseteq \iZ_g(f)$. Therefore, there exists a $B \in \mathbb{I}^f_g$ such that $\overline {d}^f_g(B)> 0$.
Let $x=\sum\limits_{n=1}^{\infty}\frac{c_n}{a_n}\in\T$ be such that $c_n=0$ whenever $n\notin B$ and for all $n\in B$, $c_n=q_n-1$ as described in Lemma \ref{theorem1}. Then applying Lemma \ref{theorem1} we can see that $x\in t^{f,g}_{(a_n)}(\T)$.
 Since $d^f_g(supp(x))\neq 0$, it follows that $supp(x)$ does not satisfy the criteria of Theorem \ref{propLaaaast} though $x\in t^{f,g}_{(a_n)}(\T)$.
}
\end{example}

Finally, following \cite{DDB}, we address the natural question as to, for an  arithmetic sequence $(a_n)$, which specific  elements of $\T$ would not surely belong to $t^{f,g}_{(a_n)}(\T)$.

\begin{proposition}\label{lemmanew01}
Let $(a_n)$ be a $q$-bounded arithmetic sequence, $f$ be an unbounded modulus function and $g\in\G$. Consider $x\in\T$ be such that
\begin{itemize}
\item[(i)] $supp(x)=\bigcup\limits_{n=1}^{\infty}[l_n,k_n]$, where $l_n,k_n\in\N$, $l_n\leq k_n < l_{n+1}-1$ for all $n\in \N$;
\item[(ii)] $\overline{d}^f_g(A)>0$, where $A=\{l_n:\ n\in\N\}$.
\end{itemize}
Then $x\notin t^{f,g}_{(a_n)}(\T)$.
\end{proposition}

\begin{proof}
Let $q_n\leq M$ for some $M\in\N\setminus\{1\}$. We set $B=\{n-2:\ n\in A\}$. Therefore, for any $n\in B$, we can observe that $n+1\not\in supp(x)$ but $n+2\in supp(x)$. Here, $\overline{d}_g^f(B)>0$ by hypothesis. Now, for all $n\in B$, we have
$$
\{a_nx\}=a_n\sum\limits_{i=n+1}^{\infty}\frac{c_i}{a_i}\ = \ a_n\sum\limits_{i=n+2}^{\infty}\frac{c_i}{a_i}\ \leq \frac{a_n}{a_{n+1}}=\frac{1}{q_{n+1}}\ \leq \frac{1}{2}.
$$
But, for all $n\in B$, we also have
$$
\{a_nx\}=a_n\sum\limits_{i=n+1}^{\infty}\frac{c_i}{a_i}\ = \ a_n\sum\limits_{i=n+2}^{\infty}\frac{c_i}{a_i}\ \geq \frac{a_n}{a_{n+2}}=\frac{1}{q_{n+1}q_{n+2}}\ \geq \frac{1}{M^2}.
$$
Hence, we find a set $B\subseteq\N$ with $\overline{d}_g^f(B)>0$ such that for all $n\in B$, $\|a_nx\|\in[\frac{1}{M^2},\frac{1}{2}]$ i.e. $x\notin t^{f,g}_{(a_n)}(\T)$.
\end{proof}

\begin{corollary}\label{lemmanew1}
Let $(a_n)$ be a $q$-bounded arithmetic sequence, $f$ be an unbounded modulus function and $g\in\G$. Consider $x\in\T$ be such that
\begin{itemize}
\item[(i)] $supp(x)=\bigcup\limits_{n=1}^{\infty}[l_n,k_n]$, where $l_n,k_n\in\N$, $l_n\leq k_n < l_{n+1}-1$ for all $n\in \N$;
\item[(ii)]  there exist $m\in\N$ such that for all $n\in\N$, $|k_n-l_n|\leq m$ and $|l_{n+1}-k_n|\leq m$.
\end{itemize}

 Then $x\notin t^{f,g}_{(a_n)}(\T)$.
\end{corollary}

\begin{proof}
Let us consider $A=\{l_n:\ n\in\N\}$. If possible we assume that $d^f_g(A)=0$. We have already seen that for each fixed $i\in\N$, $d^f_g (A_i)=0$ where $A_i=\{n-i: \ n\in A\}\cap\N$. Since, $|l_{n+1}-l_n|\leq |l_{n+1}-k_n|+|k_n-l_n|\leq 2m$, it follows that $\N=\bigcup\limits_{i=0}^{2m}A_i\in\iZ_g(f)$. Therefore our assumption was wrong and we finally get $\overline{d}^f_g(A)>0$. Now, we can observe that $x$ satisfies all the conditions of Proposition \ref{lemmanew01}. Thus $x\notin t^{f,g}_{(a_n)}(\T)$.

\end{proof}

Again consider the following example.

\begin{example}
  Consider any unbounded modulus function $f$ and $g\in\G$.Let, $x=\frac{1}{p^r-1}$ (where $r\in\N\setminus\{1\}$ and $p$ is any prime) and $a_n=p^n$. Take $l_n=k_n= rn$ and $m=r$. Therefore from Proposition \ref{lemmanew01} we obtain, $x=\sum\limits_{n=1}^{\infty} \frac{1}{p^{mn}}=\frac{1}{p^m-1} \not\in t^{f,g}_{(p^n)}(\T)$. A particular example is $\frac{1}{8} \not\in t^{f,g}_{(3^n)}(\T)$.
\end{example}

 We will generalize the idea of the above example to construct $x\in\T$ in another way (different from Proposition 3.15) which will lie outside $t^{f,g}_{(a_n)}(\T)$.
\begin{proposition}\label{sufnot1}
 Let $(a_n)$ be an arithmetic sequence of integers, $f$ be an unbounded modulus function and $g\in\G$. Consider $x\in\T$ with $\overline{d}^f_g(supp(x))>0$. If there exists $ m_1,m_2\in\R$ with $0< m_1\leq m_2< \frac{1}{2}$ and $\forall \ n\in supp(x), \ ~ \ \frac{c_n}{q_n} \in [m_1,m_2]$, then $x\not\in t^{f,g}_{(a_n)} (\T)$.
\end{proposition}
\begin{proof}
 Let $x=\sum\limits_{i\in supp(x)} \frac{c_i}{a_i}$ be the canonical representation of $x$.
 We define, $B=\{(n-1)\in\N : n\in supp(x) \}$. Since  $\overline{d}^f_g(supp(x))>0$, we must have $\overline{d}^f_g(B)>0$.

 Now $\ \forall n\in B$ one has
 \begin{eqnarray*}
 \{a_nx\}=a_n\cdot\sum\limits_{\stackrel{i\in supp(x)}{i>n}}\frac{c_i}{a_i}\ &\geq& \ a_n\cdot\sum\limits_{\stackrel{i\in supp(x)}{i>n}}\frac{m_1\cdot\frac{a_i}{a_{i-1}}}{a_i}\ = \ a_n\cdot\sum\limits_{\stackrel{i\in supp(x)}{i>n}}\frac{m_1}{a_{i-1}}\\
 &=& \ a_n\cdot\sum\limits_{\stackrel{i\in B}{i\geq n}}\frac{m_1}{a_i} \ \geq \ a_n\cdot \frac{m_1}{a_n} \ = \ m_1
 \end{eqnarray*}
 and
 \begin{eqnarray*}
 \{a_nx\}=a_n\cdot\sum\limits_{\stackrel{i\in supp(x)}{i>n}}\frac{c_i}{a_i}\ &\leq& \ a_n\cdot\sum\limits_{\stackrel{i\in supp(x)}{i>n}}\frac{m_2\cdot\frac{a_i}{a_{i-1}}}{a_i}\ = \ a_n\cdot\sum\limits_{\stackrel{i\in B}{i\geq n}}\frac{m_2}{a_i}\\
 &\leq& m_2(1+\frac{a_n}{a_{n+1}}+\frac{a_n}{a_{n+2}}+\ldots )\ \leq m_2\cdot \frac{1}{(1-\frac{1}{2})}\ =\ 2{m_2}.
 \end{eqnarray*}
 Therefore $\forall n\in B, \ ~ \ \{a_nx\} \in [m_1,2{m_2}] \ ~ \ \& \ B\not\in\iZ_g(f)$ which implies $\|a_nx\|$ cannot $f^g$-statistically converge to $0$. Thus $x\not\in t^{f,g}_{(a_n)}(\T)$.
\end{proof}

 \begin{corollary}\label{Last:Corollary}
Let $(a_n)$ be an arithmetic sequence. Then for a subset $B \subseteq \N$ there exists $x\in \T$ with $supp_{(a_n)}(x) \subseteq B$ and $x\not \in t^{f,g}_{(a_n)}(\T)$ if and only if $\overline{d}^f_g(B)>0$.
\end{corollary}


\section{Non-triviality of $f^g$-characterized subgroups and some comparison results}

In this section, our main aim is to check whether this newly obtained $f^g$-characterized subgroups are really new compared to the already studied $s$-characterized subgroups and $\alpha$-characterized subgroups.

\begin{theorem}\label{gnotequal}
For any unbounded modulus function $f$, there exists $g\in\G$ such that $t^{f,g}_{(a_n)}(\T)\subsetneq t^{\alpha}_{(a_n)}(\T)$ and $t^{f,g}_{(a_n)}(\T)\subsetneq t^{s}_{(a_n)}(\T)$.
\end{theorem}
\begin{proof}
 Write the identity function as $f_1$ for brevity, take $g(n)=\log(1+n)$ for all $n\in\N$ and $g_1(n)=n^{\alpha}$ for all $n\in\N$, where $0<\alpha< 1$.  Observe that
$$
 \lim\limits_{n\to\infty}\frac{f_1(g_1(n))}{f_1(g(n))}=\lim\limits_{n\to\infty}\frac{n^\alpha}{\ln(1+n)}=\lim\limits_{n\to\infty}\alpha n^\alpha \cdot \frac{n+1}{n}=\infty
$$
and
$$
\lim\limits_{n\to\infty}\frac{f_1(n)}{f_1(g_1(n))}=\lim\limits_{n\to\infty}\frac{n}{n^\alpha}\to\infty.
$$
  Therefore from Proposition \ref{result5}, it follows that $\iZ_g\subsetneq \iZ_\alpha$. Now, in view of (Proposition 2.6 \cite{BDK}), for any unbounded modulus function $f$ we have $\iZ_g(f)\subseteq\iZ_g$. Consequently one can choose $A\in\iZ_\alpha\setminus  \iZ_g(f)$. Clearly this means $\overline{d}_g^f(A)>0$ and from Corollary \ref {Last:Corollary} it follows that $x\in\T$ with $supp(x)\subseteq A$ satisfies $x\not\in t^{f,g}_{(a_n)}(\T)$. As $supp(x)\subseteq A\in\iZ_{\alpha}$ i.e. $d_\alpha(supp(x))=0$, from Theorem \ref {propLaaaast} we can conclude that $x\in t^{\alpha}_{(a_n)}(\T)$. Thus  $t^{f,g}_{(a_n)}(\T)\subsetneq t^{\alpha}_{(a_n)}(\T)$. Similarly taking $\alpha=1$, and choosing an appropriate support we can show that  $t^{f,g}_{(a_n)}(\T)\subsetneq t^{s}_{(a_n)}(\T)$.
\end{proof}

\begin{theorem}\label{fnotequal}
There exists an unbounded modulus function $f$ such that for any $g\in\G$, $t^{f,g}_{(a_n)}(\T)\neq t^{\alpha}_{(a_n)}(\T)$ and $t^{f,g}_{(a_n)}(\T)\neq t^{s}_{(a_n)}(\T)$.
\end{theorem}
\begin{proof}
We consider $f(x)=\log (1+x)$ and any $g\in\G$. Let $A\subset\N$ be such that $|A\cap[1,n]|=\lfloor n^{\beta}\rfloor$ (where $0<\beta<\alpha<1$). \\ Then
$$
d_\alpha(A)=\lim_{n\to\infty}\frac{|(A\cap [1,n])|}{n^{\alpha}} \leq \lim_{n\to\infty}\frac{n^{\beta}}{n^{\alpha}} = \lim_{n\to\infty}\frac{1}{n^{\alpha-\beta}}=0.
$$
Since $\frac{n}{g(n)}\nrightarrow 0$, there exists $(n_k)\subseteq\N$ such that $1+g(n_k)< (1+n_k)^2$. Now, we observe that
$$
\frac{f(|(A\cap[1,n_k])|)}{f(g(n_k))} = \frac{\log(1+|(A\cap[1,n_k])|)}{\log(1+g(n_k))} \geq \frac{\log(n_k^{\beta})}{\log (1+n_k)^2} \to \frac{\beta}{2}>0 .
$$
Therefore $\overline{d}_g^f(A)>0$ and again in view of Corollary \ref {Last:Corollary}, $x\in\T$ with $supp(x)\subseteq A$ lies outside $t^{f,g}_{(a_n)}(\T)$. On the other hand $supp(x)\subseteq A$ and $d_\alpha(A)=0$ implies $x\in t^{\alpha}_{(a_n)}(\T)\subseteq t^{s}_{(a_n)}(\T)$ in view of  Theorem \ref {propLaaaast}. Hence  $t^{f,g}_{(a_n)}(\T)\neq t^{\alpha}_{(a_n)}(\T)$ and $t^{f,g}_{(a_n)}(\T)\neq t^{s}_{(a_n)}(\T)$.

But we can actually say more. For each $0<\alpha\leq1$, we can choose $\beta$ from $(0,\alpha)$ in $\mathfrak c$ many ways. Therefore, we also have $|t^{\alpha}_{(a_n)}(\T)\setminus t^{f,g}_{(a_n)}(\T)|=\mathfrak c$.
\end{proof}

\begin{theorem}\label{cnsubsetg}
For any unbounded modulus function $f$, there exists $\mathfrak c$ many $g\in\G$ such that $t^{f}_{(a_n)}(\T) \subsetneq t^{f,g}_{(a_n)}(\T)$.
\end{theorem}
\begin{proof}
Let $f$ be an unbounded modulus function. From Remark \ref{remarkc}, there exists $\mathfrak c$ many $g\in\G$ such that $\iZ(f) \subsetneq \iZ_g(f)$. Therefore, there exists $A\subseteq\N$ such that $A\in\iZ_g(f)\setminus\iZ(f)$  i.e. $d_g^{f}(A)=0$ while $\overline{d}^{f}(A)>0$. The result then follows considering $x\in\T$ with $supp(x)\subseteq A$ from Corollary \ref {Last:Corollary} and Theorem \ref {propLaaaast}.
\end{proof}

Finally we have the following observations about the relations between characterized subgroups generated by two modulus functions which provides a broad picture about these subgroups.

$\bullet$ For any two unbounded modulus functions $f_1,f_2$, there exists a family $\G_0\subseteq \G$ of cardinality $\mathfrak c$ such that $t^{f_i,g}_{(a_n)}(\T)$ is incomparable with $t^{f_j}_{(a_n)}(\T)$ for each $g\in\G_0$ and $i,j\in\{1,2\}$. Also $t^{f_i,g_1}_{(a_n)}(\T)$, $t^{f_j,g_2}_{(a_n)}(\T)$ are incomparable for $i,j\in\{1,2\}$ and any two distinct $g_1,g_2\in\G_0$.

  The result follows from Theorem \ref{tncomparableg} following the line of the proof of Theorem \ref{cnsubsetg}.

$\bullet$ For any unbounded modulus function $f$, there exist $g_1,g_2\in\G$ such that $t^{f,g_1}_{(a_n)}(\T)\subsetneq t^{f}_{(a_n)}(\T) \subsetneq t^{f,g_2}_{(a_n)}(\T)$.

 The result follows from Proposition \ref{ppnsubset} with the proof being analogous to the proof of Theorem \ref{cnsubsetg}.

In the recent article \cite{BDH} the following open problem was posed: Problem 2.1. For any arithmetic sequence $(a_n)$ and $0 < \alpha_1 < \alpha_2 <1$, is $t^{\alpha_1} _{(a_n)}(\T) \subsetneq t^{\alpha_2} _{(a_n)}(\T)$ ?

We end the section with the following result which shows that the answer to the above problem is positive.
\begin{proposition}
For any unbounded modulus function $f$, if $g_{1},g_{2}\in \G$ are such that $\frac{f(n)}{f(g_2(n))} \geq a>0$ and
$\frac{f(g_{2}(n))}{f(g_{1}(n))} \to \infty$, then $t^{f,g_1}_{(a_n)}(\T) \subsetneq
t^{f,g_2}_{(a_n)}(\T)$.
\end{proposition}

\begin{proof}
As from Proposition \ref{result5}, it follows that $\iZ_{g_{1}}(f) \subsetneq
\iZ_{g_{2}}(f)$, the rest of the proof can be done following the method of the proof of Theorem \ref{cnsubsetg}.
\end{proof}

\subsection*{Acknowledgement:}
 The authors are thankful to the referee for several valuable suggestions which have improved the presentation of the paper. The second author acknowledges the financial support  from CSIR, HRDG, Govt. of India as a Ph.D student with sanction letter number 09/096(0961)/2018-EMR-I.

\end{document}